\documentclass{article}
\usepackage{amsmath}
\usepackage{amsthm}
\usepackage{amssymb}
\usepackage[ansinew]{inputenc}
\usepackage{url}
\usepackage[all,dvips,arc,curve,color,frame]{xy}
\input xy
\xyoption{all}

\newtheorem{lemma}{Lemma}[section]

\newtheorem{theorem}[lemma]{Theorem}
\newtheorem{cor}[lemma]{Corollary}

\newtheorem{prop}[lemma]{Proposition}
\newtheorem{defi}[lemma]{Definition}

\DeclareMathOperator\dom{dom}
\DeclareMathOperator\ima{im}

\DeclareMathOperator\Ker{ker}

\newcommand{\ran}{\rangle}

\newcommand{\lam}{\lambda}

\newcommand{\ta}{\mbox{\tiny $A$}}
\newcommand{\tb}{\mbox{\tiny $B$}}
\newcommand{\tab}{\mbox{\tiny $A\!\cap\! B$}}
\newcommand{\taa}{\mbox{\tiny $A\!\cap\! A$}}
\newcommand{\tba}{\mbox{\tiny $B\!\cap\! A$}}
\newcommand{\tem}{\mbox{\tiny $\emptyset$}}
\newcommand{\ea}{e_{\!\ta}}
\newcommand{\eb}{e_{\!\tb}}
\newcommand{\eab}{e_{\!\tab}}
\newcommand{\eaa}{e_{\!\taa}}
\newcommand{\eba}{e_{\!\tba}}
\newcommand{\ma}{\mu_{\!\ta}}
\newcommand{\mb}{\mu_{\!\tb}}
\newcommand{\mab}{\mu_{\!\tab}}
\newcommand{\mem}{\mu_{\tem}}
\newcommand{\gs}{g^*}
\newcommand{\hs}{h^*}

\title{The Largest Subsemilattices of the Semigroup\\ of Transformations on a Finite Set}

\author{Jo\~ao Ara\'ujo\\
  {\small Universidade Aberta, R. Escola Polit\'{e}cnica, 147}\\
  {\small 1269-001 Lisboa, Portugal}\\{\footnotesize \&}\\
  {\small Centro de \'{A}lgebra, Universidade de Lisboa}\\
  {\small 1649-003 Lisboa, Portugal, jaraujo@ptmat.fc.ul.pt}
\and Janusz Konieczny\\
{\small Department of Mathematics, University of Mary Washington}\\
{\small Fredericksburg, Virginia 22401, USA, jkoniecz@umw.edu}}
\date{}
\begin{document}
\maketitle

\begin{abstract}
Let $T(X)$ be the semigroup of full transformations on a finite set $X$ with $n$ elements.
We prove that every subsemilattice of $T(X)$ has at most $2^{n-1}$ elements
and that there are precisely $n$ subsemilattices of size exactly $2^{n-1}$,
each isomorphic to the semilattice of idempotents of the symmetric inverse semigroup
on a set with $n-1$ elements.

\vskip 2mm

\noindent\emph{$2010$ Mathematics Subject Classification\/}. 20M20, 06A12.

\end{abstract}

\section{Introduction}
\setcounter{equation}{0}
A \emph{semilattice} is a commutative semigroup consisting entirely of idempotents. That is, a semigroup
$S$ is a semilattice if and only if for all $a,b\in S$, $aa=a$ and $ab=ba$. A semilattice
can also be defined as a partially ordered set $(S,\leq)$ such that the greatest lower
bound $a\wedge b$ exists for all $a,b\in S$. Indeed, if $S$ is a semilattice, then $(S,\leq)$, where
$\leq$ is a relation on $S$ defined by $a\leq b$ if $a=ab$, is a poset with $a\wedge b=ab$
for all $a,b\in S$. Conversely, if $(S,\leq)$ is a poset such that $a\wedge b$ exists for all $a,b\in S$,
then $S$ with multiplication $ab=a\wedge b$ is a semilattice. (See \cite[page~10]{Bi70} and \cite[Proposition~1.3.2]{Ho95}.)

The class of semilattices
forms a variety of algebras. In the lattice of the varieties of bands (idempotent semigroups),
which has been described by Birjukov \cite{Bi70}, Fennemore \cite{Fe71}, and
Gerhard \cite{Ge70}, the variety of semilattices appears just above the trivial variety (as one of the three atoms)
and below the varieties of left and right normal bands (see \cite[Figure~1]{ArKo07}).
If $S$ is an inverse semigroup (for every $a\in S$, there is a unique $a^{-1}\in S$
such that $a=aa^{-1}a$ and $a^{-1}=a^{-1}aa^{-1}$), then the set $E(S)$ of idempotents
of $S$ is a semilattice. Therefore, semilattices play important role in the theory of
inverse semigroups (see  \cite[Chapter~5]{Ho95}). They also appear in structure theorems
for other classes of semigroups, for example for completely regular semigroups \cite[Theorem~4.1.3]{Ho95}.
In the theory of partially ordered sets, semilattices provide the most important generalization
of lattices \cite[page~22]{Bi70}.

For a set $X$, we denote by $T(X)$ the semigroup of full transformations on $X$ (functions from $X$ to $X$),
and by $I(X)$ the symmetric inverse semigroup of partial one-to-one transformations on $X$ (one-to-one
functions whose domain and image are included in $X$). In both cases, the operation is the
composition of functions. (In this paper, we will write functions on the right: $xf$ rather than $f(x)$,
and compose from left to right: $x(fg)=(xf)g$ rather than $(fg)(x)=f(g(x))$.) For a semigroup $S$,
denote by $E(S)$ the set of idempotents of $S$. The set $E(I(X))$ is a semilattice, consisting
of all $e\in I(X)$ such that $\dom(e)=\ima(e)$ and $xe=x$ for all $x\in\dom(e)$.
The semilattice $E(I(X))$, viewed as a poset, is isomorphic to the poset $(\mathcal{P}(X),\subseteq)$
of the power set $\mathcal{P}(X)$ under inclusion. On the other hand, the set $E(T(X))$ is not
a semigroup if $|X|\geq3$ (if $|X|=2$, then $E(T(X))$ is a semigroup but not a semilattice). It
is therefore of interest to determine the subsets of $E(T(X))$ that are semilattices.

In 1991, Kunze and Crvenkovi\'c \cite{KuCr91} gave a criterion for a subsemilattice
of $T(X)$ to be maximal, for a finite set $X$, in terms of transitivity orders on $X$ introduced in \cite{KuCr87}.
Any subsemilattice $S$ of $T(X)$ induces a partial order $\leq$ on $X$ (transitivity order):
$x\leq y$ if $x=y$ or $x=ye$ for some $e\in S$.  Kunze and Crvenkovi\'c have proved \cite{KuCr91}
that a subsemilattice $S$ of $T(X)$ is a maximal subsemilattice if and only if the poset $(X\leq)$
satisfies certain conditions.

All these investigations have been prompted by the fact that any finite
full transformation semigroup $T(X)$ is covered by its inverse subsemigroups
\cite{schein1} (also see \cite[Theorem~6.2.4]{higgins}),
and by the still open problem, posed in \cite{schein2}, of describing the maximal inverse subsemigroups of $T(X)$. 
We note that if $X$ is infinite, then $T(X)$ is not covered by its inverse subsemigroups \cite[Exercise~6.2.8]{higgins}.

The purpose of this paper is to determine the largest subsemilattices of $T(X)$,
where $X$ is a finite set with $n$ elements. We prove that for every subsemilattice $S$ of $T(X)$,
$|S|\leq2^{n-1}$. Moreover, we exhibit the set of all subsemilattices of $T(X)$ with the maximum
cardinality of $2^{n-1}$. This set consists of $n$ semilattices, each induced by one element
of $X$ and isomorphic to the semilattice $E(I(X'))$ of idempotents of the symmetric
inverse semigroup $I(X')$, where $X'$ is a set with $n-1$ elements.

In Section~\ref{sitx}, we describe the idempotents of $T(X)$ and state some results about
commuting idempotents. In Section~\ref{set}, we define a collection $\{E_t\}_{t\in X}$ of $n$ maximal subsemilattices
of $T(X)$ with $|E_t|=2^{n-1}$ for every $t\in X$.
In the remainder of the paper, we prove that $2^{n-1}$ is the maximum cardinality that a subsemilattice
of $T(X)$ can have, and that the semilattices $E_t$ are the only subsemilattices of $T(X)$ that have
this maximum cardinality. Our argument is by induction on $n$. The crucial step is a construction of a semilattice
$S^*$ (given a subsemilattice $S$ of $T(X)$) such that $|S|\leq 2|S^*|$ and $S^*$ can be embedded in $T(X')$,
where $|X'|=n-1$. This construction is presented in Section~\ref{sinc}.
Finally, in Section~\ref{smst}, we state and prove our main results.

Throughout this paper, we fix a finite set $X$ and reserve $n$ to denote the cardinality of $X$.
To simplify the language, we will say ``semilattice in $T(X)$'' to mean ``subsemilattice of $T(X)$.''

\section{Idempotents in $T(X)$}\label{sitx}
\setcounter{equation}{0}
For $a\in T(X)$, we denote by $\ima(a)$ the image of $a$ and by $\Ker(a)=\{(x,y)\in X\times X:xa=ya\}$ the kernel of $a$.
Let $e\in T(X)$ be an idempotent with the image $\{x_1,\ldots,x_k\}$. For each $i\in\{1,\ldots,k\}$,
let $A_i=x_ie^{-1}=\{y\in X:ye=x_i\}$ and note that $A_ie=\{x_i\}$ and $x_i\in A_i$.
The collection $\{A_1,\ldots,A_k\}$ is the partition of $X$ induced by the kernel of $e$.
We will use the following notation for $e$:
\begin{equation}\label{eide}
e=(A_1,x_1\ran(A_2,x_2\ran\ldots(A_k,x_k\ran.
\end{equation}

The following result has been obtained in \cite{ArKo03} and \cite{Ko02}.

\begin{lemma}\label{lcen}
Let $e=(A_1,x_1\ran(A_2,x_2\ran\ldots(A_k,x_k\ran$ be an idempotent in $T(X)$ and let $a\in T(X)$. Then $a$ commutes with $e$
if and only if for every $i\in\{1,\ldots,k\}$, there is $j\in\{1,\ldots,k\}$ such that
$x_ia=x_j$ and $A_ia\subseteq A_j$.
\end{lemma}

\begin{cor}\label{ccen}
Let $e,f$ be idempotents in $T(X)$ such that $ef=fe$.
Then, for all $x,y\in X$, if $x\in\ima(e)$ and $yf\in xe^{-1}$, then $xf=x$.
\end{cor}

\begin{lemma}\label{l4}
Let $x_1,\ldots,x_k$, where $k\geq2$, be pairwise distinct elements of $X$. Let $e_1,\ldots,e_k$
be idempotents in $T(X)$ such that $x_ie_i=x_{i+1}$ for every $i$, $1\leq i\leq k$, where we assume that $x_{k+1}=x_1$.
Then there is $j\in\{2,\ldots,k\}$ such that $e_1e_j\ne e_je_1$.
\end{lemma}

\begin{proof}
Suppose to the contrary that $e_1e_j=e_je_1$ for every $j\in\{2,\ldots,k\}$.
Then $x_3\in\ima(e_2)$ and $x_2e_1=x_2\in x_3e_2^{-1}$, and so $x_3e_1=x_3$ by Corollary~\ref{ccen}.
Now, $x_4\in \ima(e_3)$ and $x_3e_1=x_3\in x_4e_3^{-1}$, and so $x_4e_1=x_4$ again by Corollary~\ref{ccen}.
Using the foregoing argument $k-1$ times, we obtain that $x_{k+1}e_1=x_{k+1}$, that is,
$x_1e_1=x_1$. But this is a contradiction since $x_1e_1=x_2$. The result follows.
\end{proof}

Lemma~\ref{l4} implies the following proposition, which will be crucial for the construction presented in Section~\ref{sinc}.
The proposition is of an independent interest since, in part, it states that if $S$ is a semilattice in $T(X)$,
then the intersection of all images of the elements of $S$ is not empty.

\begin{prop}\label{p5}
Let $|X|\geq2$ and let $S$ be a semilattice in $T(X)$. Then there are $t,u\in X$ with $u\ne t$ such that for every $e\in S$,
$te=t$ and $ue\in\{u,t\}$.
\end{prop}
\begin{proof}
Suppose to the contrary that there is no $t\in X$ such that $te=t$ for every $e\in S$. Construct a sequence
$y_1,y_2,y_3,\ldots$ of elements of $X$ and a sequence $f_1,f_2,f_3,\ldots$ of elements of $S$
as follows. Start with any $y_1\in X$. By our assumption, there is $f_1\in S$ such that $y_1f_1\ne y_1$.
Set $y_2=y_1f_1$. By our assumption again, there is $f_2\in S$ such that $y_2f_2\ne y_2$.
Set $y_3=y_2f_2$, and continue the construction in the same way. Since $X$ is finite, there is
$p\geq1$ such that $y_p=y_q$ for some $q>p$. Select the smallest such $p$, and then the smallest $q$
such that $q>p$ and $y_p=y_q$. Note that $q>p+1$ since, by the construction, $y_{p+1}=y_pf_p\ne y_p$.
By the minimality of $p$ and $q$, $y_p,y_{p+1},\ldots,y_{q-1}$ are pairwise distinct, $y_if_i=y_{i+1}$
for every $i\in\{p,p+1,\ldots,q-1\}$, and $x_{q-1}f_{q-1}=x_q=x_p$.
Thus, by Lemma~\ref{l4}, there is $j\in\{p+1,\ldots,q-1\}$ such that $f_pf_j\ne f_jf_p$, which is
a contradiction. Hence a desired $t$ exists.

Suppose to the contrary that there is no $u\in X$ such that $u\ne t$ and $ue\in\{u,t\}$
for every $e\in S$. Then we obtain a contradiction
in the same way as in the first part of the proof. The only difference is that we start with $y_1\ne t$
(which is possible since $|X|\geq2$) and for each $i$, we select $f_i\in S$ such that $y_if\notin\{y_i,t\}$.
Hence a desired $u$ exists, which concludes the proof.
\end{proof}

\section{A Collection of Maximal Semilattices in $T(X)$}\label{set}
\setcounter{equation}{0}
In this section, we define a collection $\{E_t\}_{t\in X}$ of $n$ isomorphic maximal semilattices in $T(X)$
such that $|E_t|=2^{n-1}$ for every $t\in X$. In Sections~\ref{sinc} and~\ref{smst}, we will prove
that $2^{n-1}$ is the maximum cardinality of a semilattice in $T(X)$ and that $\{E_t\}_{t\in X}$
is the collection of all semilattices in $T(X)$ whose cardinality is $2^{n-1}$.

\begin{defi}\label{det}
{\rm
Fix $t\in X$ and let $X_t=X-\{t\}$, so $|X_t|=n-1$. For any $A=\{x_1,\ldots,x_k\}\subseteq X_t$
(where $0\leq k\leq n-1$, so $A$ may be empty), define $\ea\in T(X)$ by
\[
\ea=(X-A,t\ran(\{x_1\},x_1\ran\ldots(\{x_k\},x_k\ran
\]
(see notation (\ref{eide})). Note that $x\ea=x$ if $x\in A$, $x\ea=t$ if $x\notin A$,
and $\ea\eb=\eab$ for all $A,B\subseteq X_t$. We define a subset $E_t$ of $T(X)$ by
\[
E_t=\{\ea:A\subseteq X_t\}.
\]
}
\end{defi}

Recall that $I(X)$ denotes the symmetric inverse semigroup of all partial one-to-one
transformations on $X$ and that $E(I(X))$ is the semilattice of idempotents of $I(X)$.
Every $\mu\in E(I(X))$ is completely
determined by its domain: if $A=\dom(\mu)$, then
$x\mu=x$ for every $x\in A$ (and $x\mu$ is undefined for every $x\notin A$). We will
denote the idempotent in $I(X)$ with domain $A$ by $\ma$. Note that for all $A,B\subseteq X$,
$\ma\mb=\mab$. The idempotent $\mem$ is the zero in $I(X)$. We agree that $I(\emptyset)=\{0\}$.

\begin{prop}\label{p2}
For every $t\in X$,
\begin{itemize}
\item[\rm(1)] $E_t$ is a maximal semilattice in $T(X)$.
\item[\rm(2)] $E_t$ is isomorphic to $E(I(X_t))$ and $|E_t|=2^{n-1}$.
\end{itemize}
\end{prop}
\begin{proof}
Let $t\in X$. For all $\ea,\eb\in E_t$, $\ea\eb=\eab\in E_t$, $\ea\ea=\eaa=\ea$, and $\ea\eb=\eab=\eba=\eb\ea$.
Thus $E_t$ is a semilattice. Let $f\in T(X)$ be an idempotent such that $f\notin E_t$. By the definition of $E_t$,
there are $y,z\in X$ such that $z\ne y$, $y\ne t$, and $zf=y$. Suppose $z=t$, so $tf=y$. Take $\ea\in E_t$ such that
$y\notin A$. Then $t(f\ea)=y\ea=t$ and $t(\ea f)=tf=y$, so $f\ea\ne\ea f$. Suppose $z\ne t$.
Take $\ea\in E_t$ such that $y\notin A$ and $z\in A$. Then $z(f\ea)=y\ea=t$ and $z(\ea f)=zf=y$,
so $f\ea\ne\ea f$. It follows that $E_t$ is a maximal semilattice in $T(X)$. We have proved (1).

Define $\phi:E_t\to E(I(X_t))$ by $\ea\phi=\ma$. Then clearly $\phi$ is a bijection and for all $\ea,\eb\in E_t$,
\[
(\ea\eb)\phi=\eab\phi=\mab=\ma\mb=(\ea\phi)(\eb\phi).
\]
Thus $\phi$ is an isomorphism. Finally, it is clear that the mapping $\ea\to A$ is a bijection
from $E_t$ onto $\mathcal{P}(X_t)$, and so $|E_t|=|\mathcal{P}(X_t)|=2^{|X_t|}=2^{n-1}$. We have proved (2).
\end{proof}

Nichols \cite{Ni76} has proved that for every $t\in X$, the set
\[
I_t=\{a\in T(X):\mbox{$ta=t$ and $|xa^{-1}|=1$ for all $x\in\ima(a)-\{t\}$}\}
\]
is a maximal inverse subsemigroup of $T(X)$. We note that $E_t$ is the semilattice of idempotents of
the inverse semigroup $I_t$.

\section{An Inductive Construction}\label{sinc}
\setcounter{equation}{0}
Throughout this section, we assume that $|X|\geq2$
and we fix a semilattice $S$ in $T(X)$ and $t,u\in X$ with $u\ne t$ such that
for all $e\in S$, $te=t$ and $ue\in\{u,t\}$. (Such $t$ and $u$ exist by Proposition~\ref{p5}.)
Our goal is to construct a semilattice $S^*$ in $T(X)$ such that $|S|\leq2|S^*|$ and $S^*$ can
be embedded in $T(X_u)$. (Recall that $X_u=X-\{u\}$, so $|X_u|=n-1$.)

\begin{defi}\label{d6}
{\rm
For $g\in S$, define $\gs\in T(X)$ by
\[
x\gs=\left\{
\begin{array}{ll}xg & \mbox{if $xg\ne u$,}\\t & \mbox{if $xg=u$.}
\end{array}
\right.
\]
We define a subset $S^*$ of $T(X)$ by $S^*=\{\gs:g\in S\}$.
}
\end{defi}

\begin{lemma}\label{l7}
$S^*$ is a semilattice in $T(X)$.
\end{lemma}
\begin{proof}
Define $\phi:S\to T(X)$ by $g\phi=\gs$. We claim that $\phi$ is a homomorphism.
Let $g,h\in S$ and let $x\in X$. Suppose $x(gh)=u$. Then $x(gh)^*=t$ and $(xg)\hs=t$. If $xg\ne u$ then $x\gs=xg$,
and so $x(\gs\hs)=(xg)\hs=t$. If $xg=u$ then $x\gs=t$, and so $x(\gs\hs)=t\hs=th=t$. Hence $x(gh)^*=x(\gs\hs)$.

Suppose $x(gh)\ne u$. Then $x(gh)^*=x(gh)$ and $(xg)\hs=(xg)h$. If $xg\ne u$
then $x\gs=xg$, and so $x(\gs\hs)=(xg)\hs=(xg)h$. If $xg=u$ then $(xg)h=t$ (since $(xg)h=uh\in\{u,t\}$
and $(xg)h\ne u$), and so $x(\gs\hs)=t\hs=th=t=(xg)h$. Hence $x(gh)^*=x(\gs\hs)$.

We have proved that $(gh)^*=\gs\hs$ for all $g,h\in S$, and the claim follows. It is clear that $\ima(\phi)=S^*$.
Hence $S^*$ is a semilattice in $T(X)$ since any homomorphic image of a semilattice is a semilattice.
\end{proof}

\begin{lemma}\label{l8}
Let $g,h\in S$ be such that $xg=u$ and $yh=u$ for some $x,y\in X$. Suppose $g\ne h$.
Then $\gs\ne\hs$.
\end{lemma}
\begin{proof}
Since $g\ne h$, there is $z\in X$ such that $zg\ne zh$.
If $zg\ne u$ and $zh\ne u$, then $z\gs=zg\ne zh=z\hs$.

Suppose $zg=u$ or $zh=u$. We may assume that $zg=u$.
Then $z\gs=t$. Since $zh\ne zg=u$, we have $z\hs=zh$.
Since $gh=hg$,
\[
(zh)g=(zg)h=uh=(yh)h=y(hh)=yh=u,
\]
which implies that $zh\ne t$ (since $tg=t\ne u$). Thus $z\gs=t\ne zh=z\hs$.

Hence $z\gs\ne z\hs$ in all cases, and so $\gs\ne\hs$.
\end{proof}

\begin{lemma}\label{l9}
$|S|\leq 2|S^*|$.
\end{lemma}
\begin{proof}
Let $A=\{g\in S:\mbox{$xg=u$ for some $x\in X$}\}$. Then, by the definition of $S^*$, we have
$S^*=(S-A)\cup A^*$, where $A^*=\{\gs:g\in A\}$. Thus $S\subseteq S^*\cup A$, and so
$|S|\leq|S^*\cup A|\leq|S^*|+|A|$. By Lemma~\ref{l8}, $|A^*|=|A|$, and so
\[
|S|\leq|S^*|+|A|=|S^*|+|A^*|\leq|S^*|+|S^*|=2|S^*|,
\]
which concludes the proof.
\end{proof}

We will now show that the semilattice $S^*$ can be embedded in $T(X_u)$.
For a function $f:A\to B$ and $A_0\subseteq A$, we denote by
$f|_{A_0}$ the restriction of $f$ to $A_0$.

\begin{defi}\label{d10}
{\rm
We define a subset $S^*_u$ of $T(X_u)$ by
\[
S^*_u=\{e\in T(X_u):\mbox{$e=\gs|_{X_u}$ for some $g\in S$}\}.
\]
Note that indeed $S^*_u\subseteq T(X_u)$ since $x\gs\ne u$ for all $g\in S$ and $x\in X$.
}
\end{defi}

\begin{lemma}\label{l11}
$S^*_u$ is a semilattice in $T(X_u)$ isomorphic to $S^*$.
\end{lemma}
\begin{proof}
Define $\phi:S^*\to T(X_u)$ by $\gs\phi=\gs|_{X_u}$. Then $\phi$ is a homomorphism
since for all $\gs,\hs\in S^*$,
\[
(\gs\hs)\phi=(\gs\hs)|_{X_u}=(\gs|_{X_u})(\hs|_{X_u})=(\gs\phi)(\hs\phi).
\]
Clearly $\ima(\phi)=S^*_u$. Suppose $\gs|_{X_u}=\hs|_{X_u}$. Then $x\gs=x\hs$ for every $x\in X_u$.
Thus $\gs=\hs$ since $u\gs=t=u\hs$. Hence $\phi$ is one-to-one, and the result follows.
\end{proof}

\section{The Maximum Cardinality Results}\label{smst}
\setcounter{equation}{0}
In this section, we prove our main results about the maximum cardinalities of semilattices in $T(X)$.

\begin{lemma}\label{l13}
Let $|X|\geq2$, $S$ be a semilattice in $T(X)$, and $t,u\in X$
be such that $u\ne t$ and for every $e\in S$, $te=t$ and $ue\in\{u,t\}$.
Suppose that $|xe^{-1}|\leq1$ for all $e\in S$ and all $x\in X-\{u,t\}$,
and that $|uf^{-1}|\geq2$ for some $f\in S$. Then $|S|<2^{n-1}$.
\end{lemma}
\begin{proof}
Define $\lam:S\to E_t$ by
\[
x(e\lam)=\left\{
\begin{array}{lll}
xe & \mbox{if $x=u$,}\\
xe & \mbox{if $x\ne u$ and $xe\ne u$,}\\
t & \mbox{if $x\ne u$ and $xe=u$,}
\end{array}
\right.
\]
where $e\in S$ and $x\in X$. Note that, indeed, $e\lam\in E_t$ since $|xe^{-1}|\leq1$ if $x\notin\{t,u\}$,
so $|x(e\lam)^{-1}|\geq1$ if $x\ne t$.
We claim that $\lam$ is one-to-one but not onto.

Let $g,h\in S$ be such that $g\lam=h\lam$ and let $x\in X$. If $x=u$, then
$xg=x(g\lam)=x(h\lam)=x\lam$. If $x\ne u$, $xg\ne u$, and $xh\ne u$, then again
$xg=x(g\lam)=x(h\lam)=x\lam$.

Let $x\ne u$. Suppose $xg=u$ or $xh=u$. We may assume that $xg=u$. Then $x(g\lam)=t$,
and so $x(h\lam)=x(g\lam)=t$. The latter implies that either $xh=t$ or $xh=u$. We claim that
$xh=t$ is impossible. Indeed, suppose to the contrary that $xh=t$. Then, since $gh=hg$,
\[
t=tg=(xh)g=(xg)h=uh=u(h\lam)=u(g\lam)=ug=(xg)g=x(gg)=xg=u,
\]
which is a contradiction since $t\ne u$. Thus $xh\ne t$, and so $xh=u$.
But then $xg=xh$.

Hence $xg=xh$ for every $x\in X$, and so $\lam$ is one-to-one. We know that there exists
$f\in S$ such that $|uf^{-1}|\geq2$. Thus there is $z\in X$ such that $z\ne u$ and $zf=u$.
Consider any $\ea\in E_t$ such that $z\in A$ and $u\notin A$.
Note that $z\ea=z$ and $u\ea=t$. We claim that $\ea\notin\ima(\lam)$.
Suppose to the contrary that $\ea=g\lam$ for some $g\in S$. Then there is no $x\in X$
such that $xg=u$. (Indeed, if such an $x$ existed, then we would have
$ug=(xg)g=x(gg)=xg=u$, and so $t=u\ea=u(g\lam)=ug=u$, which would contradict $u\ne t$.)
It then follows from the definition of $\lam$ that $g\lam=g$.
Hence $\ea=g\lam=g\in S$, and so $\ea f=f\ea$. But $z(f\ea)=u\ea=t$ and $z(\ea f)=zf=u$,
and we have obtained a contradiction since $u\ne t$. The claim has been proved,
and so $\lam$ is not onto.

Since $\lam:S\to E_t$ is one-to-one but not onto, we have $|S|<|E_t|=2^{n-1}$.
\end{proof}

We can now prove our main results.

\begin{theorem}\label{tmain}
Let $S$ be a semilattice in $T(X)$, where $X$ is a finite set with $n$ elements.
Suppose that $S\ne E_s$ for every $s\in X$. Then $|S|<2^{n-1}$.
\end{theorem}
\begin{proof}
We proceed by induction on $n$. The result is vacuously true for $n=1$.
Let $n\geq2$ and suppose the result is true for every semilattice in $T(Z)$ with $|Z|=n-1$.
By Proposition~\ref{p5}, there are $t,u\in X$
such that $u\ne t$ and for every $e\in S$, $te=t$ and $ue\in\{u,t\}$.
We consider two cases.
\vskip 1mm
\noindent{\bf Case 1.} $|xe^{-1}|\leq1$ for all $e\in S$ and all $x\in X-\{u,t\}$.
\vskip 1mm
Then, since $S\ne E_t$, there is $f\in S$ such that $|uf^{-1}|\geq2$,
and so $|S|<2^{n-1}$ by Lemma~\ref{l13}.
\vskip 1mm
\noindent{\bf Case 2.} $|yg^{-1}|\geq2$ for some $g\in S$ and some $y\in X-\{u,t\}$.
\vskip 1mm
Then there is $z\in X$ such that $z\ne y$ and $zg=y$. Note that $z\ne u$ (since $ug\in\{u,t\}$ and $y\ne u,t$).
Consider the semilattice
$S^*_u$ in $T(X_u)$ from Definition~\ref{d10}. For any $s\in X_u$, we now have
the semilattice $E_s$ in $T(X)$ and the semilattice $E_s$ in $T(X_u)$. To avoid ambiguity,
we will denote the latter by $E'_s$.

Consider $e=\gs|_{X_u}\in S^*_u$. By the definition of $\gs$ (see Definition~\ref{d6}),
$ze=z\gs=zg=y$, and so $e\notin E'_t$ (since $y\ne t$ and $z\ne y$).
Thus $S^*_u\ne E'_t$, and so $S^*_u\ne E'_s$ for every $s\in X_u$ (since $te=t$ for every $e\in S^*_u$).
Hence, by the inductive hypothesis, $|S^*_u|<2^{n-2}$. But $|S^*_u|=|S^*|$ (by Lemma~\ref{l11})
and $|S|\leq2|S^*|$ (by Lemma~\ref{l9}). Hence $|S|\leq2|S^*|=2|S^*_u|<2\cdot2^{n-2}=2^{n-1}$.
\end{proof}

The following corollary follows immediately from Theorem~\ref{tmain} and Proposition~\ref{p2}.

\begin{cor}\label{cmain}
Let $X$ be a finite set with $n$ elements. Then:
\begin{itemize}
  \item [\rm(1)] The maximum cardinality of a semilattice in $T(X)$ is $2^{n-1}$.
  \item [\rm(2)] There are exactly $n$ semilattices in $T(X)$ of cardinality $2^{n-1}$,
namely the semilattices $E_t$ ($t\in X$) from {\rm Definition~\ref{det}}.
  \item [\rm(3)] Each semilattice $E_t$ is isomorphic to the semilattice of idempotents of the symmetric inverse semigroup $I(X')$,
where $|X'|=n-1$.
\end{itemize}
\end{cor}

By Corollary~\ref{cmain}, if $S$ is a semilattice in $T(X)$ with $m=|S|$,
then $1\leq m\leq 2^{n-1}$. It is easy to see that the converse is also true: if $1\leq m\leq 2^{n-1}$,
then $m=|S|$ for some semilattice $S$ in $T(X)$. Indeed, fix $t\in X$. If $m=2^{n-1}$, then $m=|E_t|$.
Let $1<m\leq 2^{n-1}$ and suppose $m=|S|$ for some subsemilattice $S$ of $E_t$. Let $A\subseteq X_t$
be such that $\ea\in S$ and $|A|\geq|B|$ for every $B\subseteq X_t$. Then $S_1=S-\{\ea\}$ is a subsemilattice
of $E_t$ with $|S_1|=m-1$. The converse follows.

Of course, the subsemilattices of $E_t$, with the exception of $E_t$ itself, are not maximal.
We conclude the paper with the following problem.
\vskip 2mm
\noindent{\bf Problem.} Let $X$ be a finite set with $n$ elements.
Which numbers $m$, $1\leq m\leq 2^{n-1}$, can serve as the cardinalities
of maximal semilattices in $T(X)$?

\end{document}